\documentclass[12pt,reqno]{amsart} % Specifies the document style.
\usepackage{amsmath} % symbols like equation*
\usepackage{amssymb} % symbols like mathbb
\usepackage{amsthm} % \proof
\usepackage{mathtools} % symbols like norm
\usepackage{enumitem} % change item label 

\usepackage{hyperref}
\hypersetup{citebordercolor={1 1 1}, pdfborder={0 0 0}}

\usepackage{color}

\newcommand{\sub}{\subseteq}

\newcommand{\R}{\mathbb{R}}

\newcommand{\eps}{\varepsilon}
\newcommand{\diam}{\mathrm{diam}}

\newtheorem{thm}{Theorem}[section]
\newtheorem{prop}{Proposition}[section]

\newtheorem{lem}{Lemma}[section]

\newtheorem{cor}{Corollary}[section]

\newtheorem{rmk}{Remark}[section]

\DeclarePairedDelimiter{\norm}{\lVert}{\rVert}

\makeatletter
\let\oldnorm\norm
\def\norm{\@ifstar{\oldnorm}{\oldnorm*}}

\makeatother

\begin{document}

\title[Level set estimates]
{Level set estimates for strictly convex and hyperbolic functions}


\author{Xianghong Chen}

\address[Xianghong Chen]{Department of Mathematics, Sun Yat-sen University, Guangzhou, Guangdong 510275, P.R. China}
\email{chenxiangh@mail.sysu.edu.cn}

\author{Tongou Yang}
%\thanks{$*$ Corresponding author}

\address[Tongou Yang]{Department of Mathematics\\
Southern University of Science and Technology\\
Shenzhen, Guangdong 518055, P.R. China}
\email{yangto@sustech.edu.cn}

%%%%%%%%%%%%%%%%%%%%%%%%%%

\begin{abstract}
In this note, we prove uniform upper bounds for the volume of the level set 
$$\{x\in\Omega: c\le f(x)<c+\delta\},\quad c\in\mathbb R,\;\delta>0,$$ 
for strictly convex and hyperbolic functions $f$ defined on a convex domain $\Omega\subset\mathbb{R}^n$ ($n\ge 2$). 
In particular, under a Hessian lower bound $D^2f\ge I_n$, we obtain the sharp volume bound 
$$|S^{n-1}|\, r(\Omega)^{n-2}\,\delta,$$ 
where $r(\Omega)=\frac12\,{\operatorname{diam}\Omega}$.  
As an application, we derive $L^2$ estimates for oscillatory integrals with convex phases.
\end{abstract}


\maketitle

\tableofcontents
%%%%%%%%%%%%%%%%%%%%%%%%%%%%%%%%%%%%%%%%%%%%%%%%%%%%

\section{Introduction}
Given a function $f: \Omega\rightarrow\mathbb R$ defined on a convex domain $\Omega\subset\mathbb{R}^n$, the quantity $|\{x\in\Omega: c\le f(x)<c+\delta\}|$ measures the volume of points where $f$ lies within a small range of values. Understanding how this volume depends on the regularity and curvature properties of $f$ has been the subject of extensive investigation (\textit{cf.} \cite{CCW1999, PST2001, CarberyWright2002, Rogers2005, Carbery2010, Greenblatt2010, Gressman2011, Gressman2013, CarberyMMR2014, Gilula2018, Green2020arxiv, Green2021arxiv, Green2021, Steinerberger2021, Green2022thesis, BGZZ2026, PanShaoWangWu2026arxiv}, and references therein), where refined estimates under various curvature conditions arise from the study of Fourier transforms and oscillatory integrals. 

The present note contributes to this rich literature by proving sharp uniform level set estimates under two types of geometric assumptions: a uniform lower bound on the Hessian matrix $D^2f$, and uniform upper and lower bounds on $(D^2f)^2$.

The earliest level set estimates in harmonic analysis often exploited the fact that a function with a large derivative cannot remain near a constant value over a large set. This intuition is formalised in Lemma \ref{lem:gradient_lower_bound}, where, under the assumption that $|D_\nu f(x)|\ge a>0$ for a \textit{fixed} direction $\nu$, one obtains the bound
\[
|\{x\in\Omega: c\le f(x)<c+\delta\}|\le |B^{n-1}(0,r(\Omega))|\,\frac{\delta}{a},
\]
where $$r(\Omega):=\frac12\, {\operatorname{diam}\Omega}.$$
Such gradient‑based estimates have long been used in the theory of oscillatory integrals, where they serve as a simple tool to bound the decay of Fourier transforms via integration by parts. However, in many applications, the phase function may have a critical point, rendering the gradient lower bound inapplicable. This motivates the need for estimates that rely on second‑order information. \\[-0.5em]

Our first main result, Theorem \ref{thm:elliptic}, states that if $f$ is a strictly convex function that satisfies 
\[
f(x)-f(x_0)\ge Df(x_0)\cdot (x-x_0)+\frac{a}{2}|x-x_0|^2
\]
for all $x_0,x\in\Omega\sub \R^n$ $(n\ge 2)$, then
\begin{equation}
\label{intro:delta-bound}
|\{x\in\Omega: c\le f(x)<c+\delta\}|\le |S^{n-1}|\,r(\Omega)^{n-2}\,\frac{\delta}{a}.
\end{equation}
In particular, this bound holds for any twice-differentiable function $f$ that satisfies $D^2f\ge a I_n$. The constant factor in \eqref{intro:delta-bound} is optimal, as can be seen from the case $\Omega=B^n(0,1)$, $f(x)=\frac1 2|x|^2$, $c=\frac12-\delta$, $\delta\to 0$. \\[-0.5em]

Our second main result, Theorem \ref{thm:hyperbolic}, extends Theorem \ref{thm:elliptic} to functions $f\in C^2(\Omega)$ that satisfy $(D^2f)^2\ge a^2 I_n$. When $n=2$, the bound obtained is of the form $C\frac{\delta |\log \delta |}{a}$, and when $n\ge3$, the bound is of the form $C\frac{\delta}{a}$, where $C$ is a constant that depends on $f$ (which can be made explicit up to dimensional constant) but is independent of $\delta$. 
%These constants $C$ are also made explicit up to dimensional constants. 
When $n=2$, the factor $|\log\delta|$ is necessary, as can be seen from the case $f(x)=x_1^2-x_2^2, c=0, \delta\rightarrow0$.\\[-0.5em]  

The connection between level set estimates and oscillatory integrals is explored in Corollary \ref{cor:oscillatory}, which derives the $L^2$-estimate
\[
\Bigl(\int_{\mathbb{R}} |I_{\psi}(\lambda)|^2 d\lambda\Bigr)^{1/2}\le \frac{C}{a}\,\|\psi\|_{L^2(\mathbb{R})},
\]
for $I_{\psi}(\lambda)=\int_{\Omega} e^{i\lambda f(x)}\psi(f(x))\,dx$, as a direct consequence of Plancherel's theorem and the level set estimate of Theorem \ref{thm:elliptic}. 

%%%%%%%%%%%%%%%%%%%%%%%%%%

\subsection{Notation}
\begin{itemize}
\item For $n\ge 1$, denote $B^n(x_0,r)=\left\{x \in \mathbb{R}^n:|x-x_0|<r\right\}$ and $S^{n-1}=\left\{x \in \mathbb{R}^n:|x|=1\right\}$.  
\item For $E\subset\R^n$, denote by $|E|$ the $\alpha$-dimensional Hausdorff measure of $E$, where $\alpha=\dim_HE$ is the Hausdorff dimension of $E$ (when $\alpha=n$, $|E|$ agrees with the Lebesgue measure). We adopt the convention $|B^0(x_0,r)|:=1$. 
\item A subset $\Omega\subset\R^n$ is called \textit{convex} if for all $x, y\in\Omega$ and $t\in (0,1)$, we have $(1-t)x+t y\in \Omega$. An open convex subset $\Omega\subset\R^n$ is called a \textit{convex domain}. Denote the \textit{diameter} of $\Omega$ by 
$$\diam\,\Omega=\sup_{x,y\in\Omega}|x-y|.$$
\item Let $\Omega\subset\R^n$ be a convex domain. A function $f$ on $\Omega$ is said to be \textit{convex} if for all $x, y\in\Omega$ and $t\in (0,1)$, 
\begin{equation*}
    f((1-t) x+t y)\le (1-t)f(x)+t f(y).
\end{equation*} 
$f$ is said to be \textit{strictly convex} if the inequality is strict whenever $x\neq y$. 
A vector $Df(x_0)\in\R^n$ is called a \textit{subgradient} of $f$ at $x_0\in\Omega$, if there exists a neighbourhood $U$ of $x_0$ such that for every $x\in U$, $f(x)-f(x_0)\ge Df(x_0)\cdot (x-x_0).$
\item  We use $I_n$ to denote the $n\times n$ identity matrix. For an $m\times n$ matrix $A$, let $|A|$ denote its $\ell^2\rightarrow\ell^2$ operator norm. For a matrix-valued function $A: \Omega\rightarrow\R^{m\times n}$, denote 
$$\|A\|_\infty=\sup_{x\in\Omega}|A(x)|.$$
\item We write  $a\lesssim b$ or $a=O(b)$ to indicate that there exists a constant  $C=C(n)>0$ such that $a\leq Cb$. We write $a \sim b$ to mean that both $a\lesssim b$ and $b\lesssim a$ hold. 

\item For a function $f:\Omega\sub \R^n\to \R$, we denote 
\begin{equation*}
    \|f\|_{\mathrm{Lip}}:=\sup_{x\ne y\in \Omega}\frac{|f(x)-f(y)|}{|x-y|}.
\end{equation*}
\end{itemize}

%%%%%%%%%%%%%%%%%%%%%%%%%%

\subsection{Acknowledgements}
X. C. was supported in part by 
the NNSF of China (Grant Nos. 12371105, 12426204). 
T. Y. was supported in part by 
the National Key R\&D Program of China (No. 2024YFA1015400).

\subsection{AI disclosure} No AI system was used in the conception, proof, or writing of the mathematical content of this paper, with the only exception of the mollification argument used to prove Proposition \ref{prop:mollification} suggested by Deepseek V4. ChatGPT 6-Sol was only used for the preparation of the literature review for this article.

%%%%%%%%%%%%%%%%%%%%%%%%%%%%%%%%%%%%%%%%%%%%%%%%%%%%

\section{Level set estimate via gradient lower bound}
\label{sec:gradient}
In this section, we derive level set estimates for functions whose gradient is bounded below. These estimates will be used in Section \ref{sec:nondegenerate} to derive level set estimates for functions whose Hessian is nondegenerate. \\[-0.5em]

We begin with a simple lemma. 

\begin{lem}\label{lem:gradient_lower_bound}
Let $\Omega \subset \mathbb{R}^n$ $(n\ge 1)$ be a bounded convex domain, and let $f: \Omega\rightarrow\R$ be a differentiable function. Suppose there exist $a>0$ and $\nu \in S^{n-1}$ such that
$$
\left| D_\nu f(x)\right| \ge a, \quad  \forall x \in \Omega.$$
Then for all $c\in\mathbb R$ and $\delta>0$, 
\begin{equation}
\label{eq:D_nu}
|\{x\in\Omega: c\le f(x)<c+\delta\}|\le |B^{n-1}(0,r(\Omega))|\,\frac{\delta}{a}.
\end{equation}
\end{lem}

\begin{proof}[Proof]
The case $n=1$ is straightforward and follows directly from the Mean Value Theorem. 
We therefore assume $n\ge 2$. By rotation, we may take $\nu=(0, \cdots, 0,1)$. Let 
$$
\Pi_{n-1}:\; \R^n \longrightarrow \R^{n-1},\; 
(x_1,\cdots,x_{n-1},x_n)  \longmapsto (x_1,\cdots,x_{n-1})
$$
be the orthogonal projection onto the first $n-1$ coordinates, and let $\Omega_{n-1}=\Pi_{n-1}(\Omega)$. Clearly, $\Omega_{n-1}$ is convex and satisfies 
\begin{equation}\label{eq:gradient-1}
\diam\, \Omega_{n-1}\le \diam \,\Omega. 
\end{equation}
Applying Tonelli's theorem, we can write 
$$
\begin{aligned}
&|\{x \in \Omega: c \le f(x)<c+\delta\}|\\
&=\int_{\Omega_{n-1}} \left|\left\{x_n \in \mathbb{R}:\left(x^{\prime}, x_n\right) \in \Omega,\, c \le f\left(x^{\prime}, x_n\right)<c+\delta\right\} \right| d x^\prime. 
\end{aligned}
$$
Since $\Omega$ is convex and $|D_\nu f|=\big|\frac{\partial f}{\partial x_n}\big| \ge a$, it follows from the case $n=1$ that the integrand is uniformly bounded by $\frac{\delta}{a}$. After integrating in $x^\prime$, we obtain 
\begin{equation}\label{eq:gradient-2}
|\{x \in \Omega: c \le f(x)<c+\delta\}|\le |\Omega_{n-1}|\,\frac{\delta}{a}.
\end{equation}
By the \textit{isodiametric inequality} (\textit{cf.} \cite[Theorem 8.8]{Gruber}), we have 
$$|\Omega_{n-1}|
\le |B^{n-1}(0,r(\Omega_{n-1}))|.$$
Combining this with \eqref{eq:gradient-1} and \eqref{eq:gradient-2}, the bound \eqref{eq:D_nu} follows. 
\end{proof}

\begin{rmk}
The constant factor in \eqref{eq:D_nu} is optimal, as can be seen from the case $\Omega=B^{n-1}(0,1)\times (0,\delta)$, $f(x)=x_n$, $c=0$, $\delta\to 0$. 
\end{rmk}

Using Lemma \ref{lem:gradient_lower_bound} and a suitable partition of the domain, we obtain the following. 

\begin{prop}\label{prop:gradient_lower_bound}
Let $\Omega \subset \mathbb{R}^n$ $(n\ge 2)$ be a bounded convex domain. Suppose $f\in C^1(\Omega)$ satisfies: \\
(i) $\left| D f(x) \right| \ge a$, for all $x \in \Omega$; \\
(ii) there exists $\eps_0\in(0,\diam\,\Omega]$ such that for all $x,y\in\Omega$ with $|x-y|<\eps_0$,
$$|Df(x)-Df(y)|<\frac{a}{2n}.$$
Then for all $c\in\mathbb R$ and $\delta>0$, 
\begin{equation}\label{eqn:gradient_bounded_below_general}
    |\{x \in \Omega: c \le f(x)<c+\delta\}|\lesssim\eps_0^{-1}(\diam\,\Omega)^{n}\,\frac\delta a.
\end{equation}
All implicit constants here depend only on $n$.
\end{prop}

\begin{proof}
Since $|D f(x)| \ge a$ for each $x$, there exists some $1\le i\le n$ such that $|f_{x_i}(x)|\ge a/n$. By our choice of $\eps_0$, we can cover $\Omega$ by at most $(\eps_0^{-1}\sqrt n\,\diam\,\Omega)^n$ cubes $Q_0$ of diameter $\eps_0$, such that on each $Q_0$ there is some $i=i(Q_0)$ such that $|f_{x_i}(x)|\ge a/(2n)$ for $x\in Q_0\cap \Omega$. Then \eqref{eqn:gradient_bounded_below_general} follows from applying Lemma \ref{lem:gradient_lower_bound} to each convex subset $Q_0\cap \Omega$. 
\end{proof}

\begin{rmk}
In Proposition \ref{prop:gradient_lower_bound}, if $a=1$ and $Df\in\mathrm{Lip}(\Omega)$, then one can take $\eps_0\sim \min\{\|Df\|_{\mathrm{Lip}}^{-1},$
$\diam\,\Omega\}$ and obtain
\begin{equation*}%\label{eqn:26-05-30}
    |\{x \in \Omega: c \le f(x)<c+\delta\}|\lesssim \big(\|Df\|_{\mathrm{Lip}}+(\diam\,\Omega)^{-1}\big)(\diam\,\Omega)^{n}\,\delta.
\end{equation*}
\end{rmk}

\begin{rmk}
(i) The factor $\varepsilon_0^{-1}$ in \eqref{eqn:gradient_bounded_below_general} cannot be removed under the assumptions of Proposition \ref{prop:gradient_lower_bound}. To see this, take $\Omega=(0,2\pi)^2\subset\mathbb R^2$ and consider
\[
f_N(x,y)=\frac{\sin(Nx)}{N}+y\,\varphi(\sin(Nx)),\quad (x,y)\in\Omega,
\]
where $\varphi(s)=\emph{sgn}(s)\cdot (|s|-\tfrac12)_+^{\,2}$ and $N\to\infty$. A direct verification shows that Proposition \ref{prop:gradient_lower_bound} applies with $a\sim 1$; moreover, taking $c=0$ and $\delta=1/N$, we have 
\[
\left|\{(x,y)\in\Omega:c\le f_N(x,y)<c+\delta\}\right|
\sim 1, 
\]
which shows the necessity of the factor $\varepsilon_0^{-1}$.

(ii) The factor $\varepsilon_0^{-1}$ in \eqref{eqn:gradient_bounded_below_general} remains necessary if, in addition to the assumptions of Proposition \ref{prop:gradient_lower_bound}, one assumes $|Df|\le 2a$. 
%It is unclear whether the factor $\varepsilon_0^{-1}$ in \eqref{eqn:gradient_bounded_below_general} remains necessary if, in addition to the assumptions of Proposition \ref{prop:gradient_lower_bound}, one assumes $|Df|\le 2a$. 
%Note that, by replacing $f$ with $f/a$, one may assume without loss of generality $a=1$.
This can be shown, for example, in the case $n=2$, by taking a sequence of functions $f_k$ $(k=1,2,\cdots)$ such that $\Gamma_k$, the $k^{th}$-stage Koch snowflake curve, is a level set of $f_k$. Details will appear elsewhere.
\end{rmk}

%%%%%%%%%%%%%%%%%%%%%%%%%%%%%%%%%%%%%%%%%%%%%%%%%%%%

\section{Level set estimate via Hessian lower bound}
\label{sec:convex}
In this section, we derive level set estimates for functions whose Hessian is positive definite and uniformly bounded below. \\[-0.5em]

\begin{thm}\label{thm:elliptic}
Let $\Omega\subset \R^n$ $(n\ge 2)$ be a bounded convex domain, and let $f$ be a convex function on $\Omega$ that satisfies 
\begin{equation}
\label{eq:sub-grad}
f(x) - f(x_0) \ge {D}f(x_0) \cdot (x - x_0) + \frac{a}{2} |x - x_0|
^2, \quad \forall x_0,\, x \in \Omega,
\end{equation}
where $a > 0$ is a constant and ${D}f(x_0)$ is a subgradient of $f$ at $x_0$. Then for all $c\in\mathbb R$ and $\delta>0$, 
\begin{equation}
\label{eq:Cn-prop-1}
\left|\left\{x \in\Omega: c \le  f(x)<c+\delta\right\}\right| \le  C_\Omega\,\frac{\delta}{a},
\end{equation}
where 
\begin{equation*}
    C_\Omega=|S^{n-1}|\,r(\Omega)^{n-2}.
\end{equation*}
%The constant $C_\Omega$ is sharp, as can be seen from the example $\Omega=B^n(0,1)$, $f(x)=|x|^2$, $c=1$, $\delta\to 0$.
\end{thm}  

By rescaling, the theorem is equivalent to the case $a=1$. To prove the theorem, we first consider the case of functions of two variables. The general case follows from a similar slicing argument to the proof of Lemma \ref{lem:gradient_lower_bound}. We remark that the proofs of Propositions \ref{prop:convex}, \ref{prop:elliptic} and \ref{prop:mollify} below can be easily generalised to higher dimensions $n\ge 3$, but we only choose to present the case $n=2$ for simplicity, as it is already sufficient to deduce Theorem \ref{thm:elliptic}.\\[-0.5em]

We start from the more restrictive case where $f$ has a critical point in $\Omega$.

\begin{prop}
\label{prop:convex}
Theorem \ref{thm:elliptic} holds when $n=2$, $D^2 f \ge I_2$, and $f$ has a critical point in $\Omega$. (In this case, $C_\Omega=2\pi$).
\end{prop}

\begin{proof}
By translation, we may assume $0\in \Omega$. Recall that the Minkowski functional $\rho(\cdot)$ associated with $\Omega$ is defined by 
\begin{equation}
\label{eq:rho(x)}
\rho(x)=\inf \left\{t>0: \frac{x}{t} \in \Omega\right\}, \quad x \in \mathbb{R}^2.
\end{equation}
It is easy to see that $x \in \Omega$ if and only if $\rho(x)<1$, and that $\rho(\cdot)$ is positively homogeneous, i.e., $\rho(\cdot)$ satisfies 
$$
\rho(\lambda x)=\lambda \rho(x), \quad \lambda \ge  0,\; x \in \mathbb{R}^2.
$$ 
Moreover, the boundary $\partial \Omega$ coincides 
$$
S_{\Omega}:=\left\{x \in \R^2: \rho(x)=1\right\}.
$$
Consider the map
$$
\begin{aligned}
\Pi:\; S_{\Omega} \times(0, \infty) & \longrightarrow \R^2 \backslash\{0\}, \\
(\omega, \rho) & \longmapsto \rho \omega.
\end{aligned}
$$
It is easy to verify that $\Pi$ is a homeomorphism, with the inverse given by
$$
\begin{aligned}
\Pi^{-1}:\; \R^2 \backslash\{0\} & \longrightarrow S_{\Omega} \times(0, \infty), \\
x & \longmapsto\Big(\frac{x}{\rho(x)},\;\rho(x)\Big).
\end{aligned}
$$
If we equip $S_{\Omega}$ with the measure 
$$
\mu_S(E):=2\left|\Pi(E \times(0,1))\right|, \quad E \subset S_{\Omega}, 
$$
and equip $(0, \infty)$ with the measure $\rho d\rho$, then, similar to integration in polar coordinates, we have 
\begin{equation}
\label{eq:polar}
\int_{\R^2} g(x) d x=\int_{S_{\Omega}}\left(\int_0^{\infty} g(\rho \omega)\rho d \rho \right)\mu_S(d \omega). 
\end{equation} 

We now turn to the proof of \eqref{eq:Cn-prop-1} with $C_\Omega=2\pi$. After replacing $f$ by $f-f(0)$, 
we may assume that $f(0)=0$. Fix $x \in \Omega$. 
Since $Df(0)=0$, applying Taylor’s theorem with a Lagrange remainder to $f(tx)$ at $t=0$ yields
$$
f(x)=\frac{1}{2} x^{T} D^2 f(\theta x)\,x 
$$
for some $\theta=\theta(x)\in(0,1)$. 
Since $D^2 f \ge  I_2$, it follows that 
\begin{equation}
\label{eq:eq1}
f(x)\ge  \frac{1}{2}|x|^2\ge 0, \quad x \in \Omega. 
\end{equation} 
In particular, $f(x)=0$ holds only when $x=0$. 

In view of \eqref{eq:eq1}, it suffices to consider the case $c>0$. 
Using \eqref{eq:polar}, we can write
\begin{equation}
\label{eq:eq4}
\begin{aligned}
    &\big|\left\{x\in\Omega: c\le f(x) < c+\delta\right\}\big|\\
=&\int_{S_\Omega}\left(\int_{\{\rho\in(0,1):\,c \le  f(\rho \omega) < c+\delta\}} \rho d \rho\right) \mu_S(d\omega). 
\end{aligned}
\end{equation}
Writing 
$$\big\{\rho\in(0,1): c \le  f(\rho\omega) < c+\delta\big\}=:\big[\rho_0(\omega), \rho_1(\omega)\big),$$
with 
$$f\big(\rho_0(\omega)\omega\big)=c,\quad f\big(\rho_1(\omega)\omega\big)\le c+\delta,$$ 
we have 
\begin{equation}
\label{eq:eq2}
\int_{\{\rho\in(0,1):\,c \le  f(\rho \omega) < c+\delta\}} \rho  d \rho = \frac{\rho_1(\omega)^n-\rho_0(\omega)^n}{n}. 
\end{equation}
Now apply Cauchy's Mean Value Theorem to the functions $\rho^2$ and $\phi(\rho):=f(\rho\omega)$ at $\rho=\rho_0(\omega),\, \rho_1(\omega)$. We obtain 
\begin{align}
\frac{\rho_1^2-\rho_0^2}{\delta}
&\le \frac{\rho_1^2-\rho_0^2}{f(\rho_1\omega)-f(\rho_0\omega)}
= \frac{2 r }{\phi'(r)}\label{eq:eq3},
\end{align}
for some $r \in(\rho_0, \rho_1)\subset(0,1)$. 
However, notice that 
\begin{equation*}
    \phi'(r)=\phi'(r)-\phi'(0)=\phi''(\theta r)r=\omega^T D^2f(\theta r\omega)\omega r\ge |\omega|^2 r.  
\end{equation*}
Therefore, 
$$\frac{2r }{\phi'(r)}
\le \frac{2 r }{|\omega|^2 r}=\frac{2}{|\omega|^2}.$$ 
Combining this with \eqref{eq:eq4}, \eqref{eq:eq2}, and \eqref{eq:eq3}, 
we obtain 
\begin{equation*}
\big|\left\{x\in\Omega: c\le f(x) < c+\delta\right\}\big|
\le C_\Omega \delta, 
\end{equation*}
where 
\begin{align}
C_\Omega 
&=\int_{S_\Omega}\frac{1}{|\omega|^2}\mu_S(d\omega)\notag\\
&=\int_0^1\left(\int_{S_\Omega}\frac{\rho}{|\rho\omega|^2}\mu_S(d\omega)\right)\rho d\rho\notag\\
&=\int_{\Omega}\frac{\rho(x)}{|x|^2}dx.\label{eq:C_Omega-rho}
\end{align}
Using the polar coordinates, we can write 
$$\partial \Omega=\big\{(\omega,r(\omega)):\omega\in S^{1}\big\}.$$
Then, for $x=|x|\omega\in\Omega\backslash\{0\}$, we have 
$\rho(x)=\frac{|x|}{r(\omega)}.$
Thus, from \eqref{eq:C_Omega-rho} we obtain 
\begin{align*}
C_\Omega
&=\int_{S^1}\int_0^{r(\omega)} \frac{r(\omega)^{-1}}{r} rdrd\omega= \int_{ S^1}d\omega=2\pi.
\end{align*}
This completes the proof of Proposition \ref{prop:convex}. 
\end{proof}

We now consider a slightly more general version of Proposition \ref{prop:convex} where $f$ does not necessarily have a critical point in $\Omega$. 

\begin{prop}\label{prop:mollification}
\label{prop:elliptic}
Theorem \ref{thm:elliptic} holds when $n=2$ and $D^2 f\ge I_2$.
\end{prop}

\begin{proof}
By translation, we may assume that $0\in\Omega$. 
Let $\rho(\cdot)$ be the Minkowski functional \eqref{eq:rho(x)} associated with $\Omega$, and denote for $\rho>0$ 
\begin{equation}
\label{eq:Omega_rho}
\Omega_\rho=\{x\in\R^2: \rho(x)<\rho\}.
\end{equation}
Fix $\varepsilon\in(0,1/2)$. 
Consider 
$$g_{\varepsilon,A}(x)=g_{\varepsilon,A}(\rho(x)
),\quad x\in\R^2,$$
where $A>0$ is a large constant (to be chosen), and 
\[g_{\varepsilon,A}(\rho)=\begin{cases}
A\big(\rho-(1 - \varepsilon)\big)^{2},&\rho\in [1-\varepsilon,\infty); 
\\ 0, & \rho\in[0,1-\varepsilon].\end{cases}\]
It is easy to see that $g_{\varepsilon,A}(\rho)$ is convex and nondecreasing in $\rho\in[0,\infty)$. Since $\rho$ is convex, $g_{\varepsilon,A}(x)$ is convex in $x\in\R^2$. 

Fix a nonnegative bump function $\varphi \in C_c^\infty(B^2(0,1))$ with $\int_{\mathbb{R}^2} \varphi(x) dx = 1$, and denote $\varphi_a(x) = a^{-2} \varphi(x/a)$ $(a>0)$. 
Consider the mollified function 
$$\widetilde g_{\varepsilon,A}(x)=(g_{\varepsilon,A}*\varphi_a)(x),\quad x\in\R^2.$$
It is easy to verify that, when $a>0$ is small enough (depending only on $\varepsilon$ and $\Omega$), we have 

\noindent (i) $\widetilde g_{\varepsilon,A}\in C^2(\R^2)$;  

\noindent (ii) $D^2\widetilde g_{\varepsilon,A}(x)\ge 0,\;x\in\R^2$; 

\noindent (iii) $\widetilde g_{\varepsilon,A}(x)= 0,\; x\in\Omega_{1-2\varepsilon}$; 

\noindent (iv) $\widetilde g_{\varepsilon,A}(x)\ge A(\varepsilon/2)^2,\; x\in\partial\Omega_{1-\varepsilon/4}$. 

We can use $\widetilde g_{\varepsilon,A}$ to ``blow up'' the value of $f(x)$ outside ${\Omega}_{1 - 2\varepsilon}$ as follows: Set 
$$f_{\varepsilon,A}(x)=f(x)+\widetilde g_{\varepsilon,A}(x)
,\quad x\in\Omega.$$
From (i)--(iii), we have 

\noindent(i$^\prime$)
$f_{\varepsilon, A} \in C^{2}(\Omega)$; 

\noindent(ii$^\prime$)
${D}^{2} f_{\varepsilon, A}(x)\ge  I_{2},\;x\in\Omega$; 

\noindent(iii$^\prime$)
$f_{\varepsilon,A}(x)=f(x),\; x\in{\Omega}_{1-2\varepsilon}.$\\
Moreover, choosing 
$$A=3\left(\frac{\varepsilon}{2}\right)^{-2} \|f\|
_{C(\overline{\Omega}_{1-\frac{\varepsilon}{4}})},$$
from (iv) we have, 
for $x\in\partial\Omega_{1-\varepsilon/4}$, 
$$
\begin{aligned}
f_{\varepsilon, A}(x) &= \widetilde g_{\varepsilon, A}\left(x\right)+f(x) \\
&\ge 3 \|f\|
_{C(\overline{\Omega}_{1-\frac{\varepsilon}{4}})}+f(x) \\
&\ge  2\|f\|
_{C(\overline{\Omega}_{1-\frac{\varepsilon}{4}})} \\
&>\|f\|
_{C(\overline{\Omega}_{1-{2\varepsilon}})}\\
&=\|f_{\varepsilon, A}\|
_{C(\overline{\Omega}_{1-{2\varepsilon}})}.
\end{aligned}
$$
As a consequence, the minimum $$\min _{x \in \overline{\Omega}_{1-\frac{\varepsilon}{4}}} f_{\varepsilon, A}(x)$$ is attained in ${\Omega}_{1-\frac{\varepsilon}{4}}$, which yields a critical point of $f_{\varepsilon, A}$ in $\Omega$. 
Now apply the case (with a critical point) we have shown above to $f_{\varepsilon,A}$. 
We have  
\[
\left| \left\{ x \in \Omega: c \le f_{\varepsilon,A}(x) < c+\delta \right\} \right| \le 2\pi \delta.
\]
It follows from (iii$^\prime$) that 
\[
\left| \left\{ x \in \Omega_{1-2\varepsilon}: c \le f(x) < c+\delta \right\} \right| \le 2\pi \delta.
\]
Finally, taking $\varepsilon \to 0$ gives \eqref{eq:Cn-prop-1}. 
\end{proof}

Our next step is to further relax the regularity assumption $f \in C^{2}(\Omega)$ using a standard mollification argument. 

\begin{prop}\label{prop:mollify}
    Theorem \ref{thm:elliptic} holds when $n=2$.
\end{prop}

\begin{rmk}
When $n=2$, by a limiting argument, \eqref{eq:Cn-prop-1} holds even when $\Omega$ is unbounded.
\end{rmk}

\begin{proof}[Proof of Proposition \ref{prop:mollify}]
By a translation, assume $0\in\Omega$. Denote by $\rho(\cdot)$ the Minkowski functional associated with $\Omega$, and denote $\Omega_\rho$ as in \eqref{eq:Omega_rho}. 
Upon taking $\varepsilon \to 0$, it suffices to show that for any $\varepsilon\in(0,1)$, we have
$$|\{x \in \Omega_{1-\varepsilon}: c < f(x) < c + \delta\}|
 \le 2\pi \delta.$$
Let $0 \le \varphi \in C_c^\infty(\mathbb R^2)$ be a bump function with $\int_{\mathbb R^2} \varphi(x) dx = 1$.
Denote $\varphi_k(x) = k^2 \varphi(k x)$ and  $f_k = f * \varphi_k$. 
Since
$$f(x) = \lim_{k \to \infty} f_k (x), \quad x \in \Omega_{1-\varepsilon},$$
by Fatou's lemma, it suffices to show that
\begin{equation}
\label{eq:fk-bound}
|\{x \in \Omega_{1-\varepsilon} : c < f_k (x) < c+\delta\}| \le C_\Omega \delta
\end{equation}
holds for sufficiently large $k$. 

Since $f_k\in C^2(\Omega_{1-\varepsilon})$, 
applying Proposition \ref{prop:elliptic} to $f_k$ would give 
$$|\{x \in \Omega_{1-\varepsilon}: c < f_k (x) < c+\delta\}| \le 2\pi\delta,$$
if we could show 
$$D^2 f_k(x) \ge I_2, \quad x \in \Omega_{1-\varepsilon},$$
or, equivalently,
\begin{equation}
\label{eq:2nd-der}
\lim_{t \to 0} \frac{f_k(x+t\omega) + f_k(x-t\omega) - 2 f_k(x)}{t^2} \ge 1, 
\end{equation}
for any $x \in \Omega_{1-\varepsilon}$ and $\omega \in S^{1}$. 
To show \eqref{eq:2nd-der}, notice that by \eqref{eq:sub-grad}, 
\begin{align*}
&f(x_0+t\omega) + f(x_0-t\omega) - 2 f(x_0) \\
=&\big(f(x_0+t\omega) -  f(x_0)\big) + \big(f(x_0-t\omega) -  f(x_0)\big)\\
\ge &\big(D f(x_0) \cdot (t\omega) + \frac{1}{2} |t\omega|^2\big) + \big(D f(x_0) \cdot (-t\omega) + \frac{1}{2} |-t\omega|^2\big) \\
=& t^2, 
\end{align*}
when $x_0\in \Omega$ and $t$ is sufficiently small. Therefore, for $x \in \Omega_{1-\varepsilon}$ and sufficiently small $t$, we have 
\begin{align*}
&f_{k}(x + t\omega)+f_{k}(x - t\omega)-2f_{k}(x) \\
=&\int_{\mathbb{R}^{n}}\left[f(x - y + t\omega)+f(x - y - t\omega)-2f(x - y)\right]\varphi_{k}(y)dy \\
\geq&\int_{\mathbb{R}^{n}}t^{2}\varphi(y)dy 
=t^{2}, 
\end{align*}
from which \eqref{eq:2nd-der} follows. This completes the proof of Proposition \ref{prop:mollify}.
\end{proof}

Finally, we use a similar slicing argument to the proof of Lemma \ref{lem:gradient_lower_bound} to prove Theorem \ref{thm:elliptic}. In fact, we can prove something stronger. It is easy to see that Theorem \ref{thm:elliptic} follows from Proposition \ref{prop:slicing} as a corollary.
\begin{prop}\label{prop:slicing}
Let $\Omega\subset \R^n$ $(n\ge 2)$ be a bounded convex domain, and let $f:\Omega\to \R$ be a measurable function. For $x\in \Omega$, write $x=(x',x'')$, where $x'\in \R^2$ and $x''\in \R^{n-2}$, and denote by $\Omega'$ the orthogonal projection of $\Omega$ onto the first two coordinates. Fix $x''$ and denote $g(x')=f(x)$. Suppose there exists $a>0$ such that 
\begin{equation}
\label{eq:sub-grad_slicing}
g(x') - g(x'_0) \ge {D}g(x'_0) \cdot (x' - x'_0) + \frac{a}{2} |x' - x'_0|
^2, \quad \forall x'_0,\, x' \in \Omega',
\end{equation}
where $a > 0$ is a constant and ${D}g(x'_0)$ is a subgradient of $g$ at $x'_0$. Then for all $c\in\mathbb R$ and $\delta>0$, 
\begin{equation}\label{eqn:slicing}
\left|\left\{x \in\Omega: c \le 
f(x)<c+\delta\right\}\right| \le  
 |S^{n-1}|\,
r(\Omega)^{n-2}\,\frac{\delta}{a}.
\end{equation}
\end{prop}

\begin{proof}
    Let $\Pi:\R^n\to \R^{n-2}$ be the orthogonal projection onto the last $n-2$ coordinates, and let $\Omega''=\Pi(\Omega)$. Clearly, $\Omega''$ is convex and $\diam\,\Omega''\le \diam\,\Omega$. Applying Tonelli's theorem and Proposition \ref{prop:mollify} to $g(x')$, we have
    \begin{equation*}
    \begin{aligned}
        |\{x\in \Omega:c\le f(x)<c+\delta\}|
        &\le \int_{\Omega''}|\{x':c\le f(x)<c+\delta\}|dx'\\
        &\le |\Omega''|\,2\pi\,\frac{\delta}{a}.
    \end{aligned}
    \end{equation*}
    By the isodiametric inequality, we have
    \begin{align*}
        |\Omega''|&\le |B^{n-2}(0,r(\Omega''))|\\
        &\le|B^{n-2}(0,r(\Omega))|\\
        &=|B^{n-2}(0,1)|\,r(\Omega)^{n-2}.
        %=|B^{n-2}(0,1)|\left(\frac {\diam\,\Omega}2\right)^{n-2}\delta.
    \end{align*}
    Then \eqref{eqn:slicing} follows from the identity $2\pi |B^{n-2}(0,1)|=|S^{n-1}|$.
\end{proof}

%%%%%%%%%%%%%%%%%%%%%%%%%%%%%%%%%%%%%%%%%%%%%%%%%%%%

\section{Level set estimate via nondegenerate Hessian}
\label{sec:nondegenerate}
In this section, we derive level set estimates for functions $f$ whose Hessian is nondegenerate but not necessarily positive definite. \\[-0.5em]

Recall Proposition \ref{prop:slicing}, which requires that the projection of $f$ onto a fixed subspace be strictly convex. If we only assume that $|\det D^2 f|$ is bounded above and below, or equivalently, $(D^2 f)^2$ is bounded above and below, then the situation becomes more subtle, as stated in the theorem below.

\begin{thm}\label{thm:hyperbolic}
Let $\Omega\subset \R^n$ $(n\ge 2)$ be a bounded convex domain. Suppose $f\in C^2(\Omega)$ satisfies:\\
(i) $(D^2 f)^{-1}$ exists and is bounded on $\Omega$;\\
(ii) There exists $\eps_0\in(0,\diam \,\Omega]$ such that for all $x,y\in\Omega$ with $|x-y|<\eps_0$, we have
    \begin{equation*}
        |D^2 f(x)-D^2 f(y)|<\frac{1}{2\|(D^2 f)^{-1}\|_\infty}.
    \end{equation*}
Then for all $c\in\R$ and 
    \begin{equation}\label{eqn:delta_upper_bound}
    0<\delta<\min\left\{\frac 1 2,\,\frac{\|D^2 f\|_\infty \|(D^2 f)^{-1}\|_\infty^2}{\eps_0^2}\right\}, 
     %\delta<\begin{cases}
        % \frac 1 2,\quad & \text{if }n\ge 3,\\
         %\min\left\{\frac 1 2,\,\frac{\|D^2 f\|_\infty \|(D^2 f)^{-1}\|_\infty^2}{\eps_0}\right\}, & \text{if }n=2,
     %\end{cases}
    \end{equation}
    we have \begin{equation}\label{eqn:hyperbolic_main_estimate}
    \begin{aligned}
        \left|\{x\in \Omega: c\le  f(x)< c+\delta\}\right|
        \lesssim  
        \eps_0^{-2}\|D^2 f\|_\infty\|
            (D^2 f)^{-1}\|_\infty^2(\diam\,\Omega)^n  m(\delta),
    \end{aligned}
    \end{equation}
where
    \begin{equation}          
        m(\delta)=\begin{cases}
            \delta |\log \delta|, &\text{if }n=2,\\
            \delta, & \text{if }n\ge 3.
        \end{cases}
    \end{equation}    
The implicit constant above (and in the proof below) depends only on $n$.
\end{thm}

\begin{rmk}
    (i) A common situation is when $\diam\,\Omega \gtrsim 1$, $\|D^2 f\|_\infty\gtrsim 1$, $\|(D^2 f)^{-1}\|_\infty\gtrsim 1$ and $D^2f$ is Lipschitz with norm $\gtrsim 1$. Then one can take $\eps_0^{-1}\sim \|(D^2 f)^{-1}\|_\infty \|D^2 f\|_{\mathrm{Lip}}$, whence the right-hand side of \eqref{eqn:hyperbolic_main_estimate} can be taken to be
    \begin{equation}\label{eqn:26-06-11}
        \|D^2 f\|^2_{\mathrm{Lip}}\|D^2 f\|_\infty\|(D^2 f)^{-1}\|_\infty^4(\diam\,\Omega)^n m(\delta).
    \end{equation}
    When $m(\delta)=\delta$, one can check that the term \eqref{eqn:26-06-11} is invariant with respect to rescalings $\lambda f(\mu x)$ for $\lambda,\mu>0$.\\
    (ii) If we have a uniform lower bound $|\det D^2 f|\gtrsim 1$ on $\Omega$, then by the adjugate matrix formula, the right-hand side of \eqref{eqn:hyperbolic_main_estimate} can be taken to be
    \begin{equation}
    \begin{aligned}
        &\eps_0^{-2}\|D^2 f\|^{2n-1}_\infty(\diam\,\Omega)^n m(\delta).
    \end{aligned}
    \end{equation}\\
    (iii) The factor $|\log \delta|$ is necessary when $n=2$, as can be seen from the case $\Omega=(-1,1)^2$ and 
    $$f(x)=x_1^2-x_2^2,\;\;c=0,\;\;\delta\rightarrow0.$$ 
\end{rmk}

\begin{proof}[Proof of Theorem \ref{thm:hyperbolic}]
By a translation we may assume $c=0$. By the inverse function theorem, $Df$ restricted to any cube $Q_0$ of side length $\eps_0$ is invertible. Cover $\Omega$ by a grid consisting of at most $O(\eps_0^{-1}\mathrm{diam}\, \Omega)^n$ such cubes $Q_0$. It then suffices to show that for each fixed $Q_0$, we have
\begin{equation*}
\begin{aligned}
    &\left|\{x\in Q_0\cap \Omega: 0\le  f(x)\le   \delta\}\right|  \\
    \lesssim &\|D^2 f\|_\infty\|(D^2 f)^{-1}\|_\infty^2\, \eps_0^{n-2} m(\delta).
\end{aligned}    
\end{equation*}
Denote by $q\in Q$ a minimiser of $|Df(q)|$, and denote 
\begin{equation}\label{eqn:defn_r_0}
    r_0:=\big(\|D^2 f\|_\infty\|(D^2 f)^{-1}\|_\infty^2\, \eps_0^{n-2}m(\delta)\big)^{1/n},
\end{equation}
so that
\begin{align*}
    |B^n(q,r_0)|\sim r_0^n=  \|D^2 f\|_\infty\|(D^2 f)^{-1}\|_\infty^2\, \eps_0^{n-2}m(\delta).
\end{align*}
If $r_0\ge \eps_0$ then we are done. If $r_0\le \eps_0$, it then suffices to show that
\begin{equation*}
\begin{aligned}
    &\left|\{x\in Q_0\cap \Omega: |x-q|\ge r_0,\,0\le  f(x)\le   \delta\}\right| \\
    \lesssim &\|D^2 f\|_\infty\|(D^2 f)^{-1}\|_\infty^2 \,\eps_0^{n-2} m(\delta).
\end{aligned}    
\end{equation*}
To this end, we perform a dyadic decomposition according to $\sigma\le |x-q|\le 2\sigma$ for dyadic numbers $\sigma$ essentially between $r_0$ and $\eps_0$. It then suffices to estimate for each $\sigma$
\begin{equation}\label{eqn:sum_over_sigma}
     \sum_{\stackrel{\sigma \,\,\text{dyadic}}{r_0\lesssim \sigma \lesssim \eps_0}}\left|\{x\in Q_0\cap \Omega: |x-q|\in [\sigma,2\sigma],\,0\le  f(x)\le   \delta\}\right|.
\end{equation}
Now fix $\sigma$. Since the set $\{x\in Q_0\cap \Omega: |x-q|\sim \sigma\}$ can be covered by $O(1)$ many cubes $Q$ of side length $\sim\sigma$, it then suffice to estimate for each $Q$ the quantity
\begin{equation*}
    |\{x\in Q\cap \Omega:0\le  f(x)\le   \delta\}|.
\end{equation*}

To this end, we apply a rescaling. Since $\Omega$ is convex, so is $Q\cap \Omega$. Denote by $\lambda_Q$ an affine bijection from $[-1,1]^2$ to $Q$, and denote $g=f\circ \lambda_Q$. Invoking the invertibility of $Df$, when $|x-q|\sim 
\sigma$, we have $|Df(x)-Df(q)|\gtrsim \|(D^2 f)^{-1}\|^{-1}_\infty \sigma$; in particular, $|Df(x)|\gtrsim \|(D^2 f)^{-1}\|^{-1}_\infty\sigma$ by the minimality of $|Df(q)|$. This means that $|Dg|\gtrsim \|(D^2 f)^{-1}\|^{-1}_\infty\sigma^2$ on the convex subset $[-1,1]^2\cap \lambda_Q^{-1}(\Omega)$. Also, $\|Dg\|_{\mathrm{Lip}}\lesssim \sigma^2 \|D^2 f\|_\infty$. Applying Proposition \ref{prop:gradient_lower_bound} to $g$ with an obvious rescaling, we have
\begin{equation*}
    \left|\{x\in [-1,1]^2:0\le  g(x)\le   \delta\}\right|  \lesssim \|D^2 f\|_\infty \|(D^2 f)^{-1}\|_\infty^2\sigma^{-2}\delta.
\end{equation*}
Rescaling back, we obtain
\begin{equation*}
    |\{x\in Q\cap \Omega:0\le  f(x)\le   \delta\}|\lesssim \|D^2 f\|_\infty \|(D^2 f)^{-1}\|_\infty^2\sigma^{n-2}\delta.
\end{equation*}
We can bound the sum in \eqref{eqn:sum_over_sigma} above by a dimensional constant times
\begin{align*}
    &\sum_{\stackrel{\sigma \,\,\text{dyadic}}{r_0\lesssim \sigma \lesssim \eps_0}}\|D^2 f\|_\infty \|(D^2 f)^{-1}\|_\infty^2\sigma^{n-2}\delta\\
    \lesssim & 
    \begin{cases}
        \|D^2 f\|_\infty\|(D^2 f)^{-1}\|_\infty^2 \delta \eps_0^{n-2},\quad & n\ge 3\\
        \|D^2 f\|_\infty\|(D^2 f)^{-1}\|_\infty^2 \delta \log (\eps_0/r_0),\quad &n=2.
    \end{cases}.
\end{align*}
Note that when $n=2$, the choices of $r_0$ in \eqref{eqn:defn_r_0} and $\delta$ in \eqref{eqn:delta_upper_bound} allow us to conclude that $\log (\eps_0/r_0)\le |\log \delta|$. Summing over cubes $Q_0$, we obtain the required bound. 
\end{proof}

%%%%%%%%%%%%%%%%%%%%%%%%%%%%%%%%%%%%%%%%%%%%%%%%%%%%

\section{An application to oscillatory integrals}
\label{sec:application}
In this section, as a direct application of Theorem \ref{thm:elliptic}, we derive the following oscillatory integral estimate. 

\begin{cor}\label{cor:oscillatory}
Let $\Omega$ and $f$ be as in Theorem \ref{thm:elliptic}. For $\psi \in L^2(\mathbb{R})$, define 
$$
I_\psi(\lambda)=\int_{\Omega} e^{i \lambda f(x)} \psi(f(x)) d x, \quad \lambda \in \mathbb{R} .
$$
Then we have
$$
\left(\int_\R|I_\psi(\lambda)|^2 d \lambda\right)^{1/2} \le  \sqrt{2\pi}\,\frac{C_\Omega}{a}\|\psi\|_{L^2(\mathbb{R})}, 
$$
where $C_\Omega$ is as in Theorem \ref{thm:elliptic}.
\end{cor}

\begin{proof}
For convenience, we regard \(\Omega\) as a probability space equipped with the probability measure \(\frac{1}{|\Omega|} dx\), and regard
\[
\xi = f(x), \quad x \in \Omega
\]
as a random variable. 
It follows from Theorem \ref{thm:elliptic} that \(\xi\) has a probability density function \(p(\xi)\) that satisfies
\begin{equation}
\label{eq:p-bound}
0\le p(\xi) \leq \frac{1}{|\Omega|}\frac{C_\Omega}{a}, \quad \xi \in \mathbb{R},
\end{equation}
where $C_\Omega$ is as in Theorem \ref{thm:elliptic}. 
Therefore, we can rewrite
\[
I_\psi(\lambda) = |\Omega|\cdot\mathbb{E}\big[ e^{i\lambda\xi} \psi(\xi) \big]
= |\Omega| \int_\R e^{i\lambda\xi} \psi(\xi) p(\xi) \, d\xi.
\]
Notice that the last integral is the (inverse) Fourier transform of \(\psi p\). Thus, by Plancherel's theorem, we have 
\[
\left( \int_\R |I_\psi(\lambda)|^2 \, d\lambda \right)^{\frac{1}{2}} = |\Omega| \sqrt{2\pi} \| \psi  p\|_{L^2}^2.
\]
Combining this with \eqref{eq:p-bound}, we obtain
\[
\| I_\psi \|_{L^2(\mathbb{R})}\leq 
\sqrt{2\pi}\,\frac{C_\Omega}{a} \|\psi \|_{L^2(\mathbb{R})},
\]
as desired. 
\end{proof}

\begin{rmk}
Similar estimates can be derived for $\Omega$ and $f$ satisfying the conditions of Proposition \ref{prop:gradient_lower_bound} and Theorem \ref{thm:hyperbolic} (with $n\ge 3$). 
\end{rmk}
%%%%%%%%%%%%%%%%%%%%%%%%%%%%%%%%%%%%%%%%%%%%%%%%%%%%

\bibliographystyle{alpha}
\bibliography{ref}

\end{document}